\documentclass{article}
\usepackage[affil-it]{authblk}
\usepackage[utf8]{inputenc}
\usepackage{amsmath, amsfonts, amsthm, amssymb, latexsym}
\usepackage{geometry}
\usepackage{enumerate}
\usepackage{multirow}
\usepackage{rotating}
\usepackage{graphicx,color}
\usepackage{float}
\usepackage{mathrsfs}

\newtheorem{thm}{Theorem}[section]

\newtheorem{prop}[thm]{Proposition}

\newtheorem*{remark}{Remark}

\title{An inverse problem for a fractional diffusion equation with fractional power type nonlinearities}
\author{Li Li}
\affil{Institute for Pure and Applied Mathematics, University of California,\\
Los Angeles, CA 90095, USA}
\date{}

\begin{document}
\maketitle

\noindent \textbf{ABSTRACT.}\, 
We study the well-posedness of a semi-linear fractional diffusion equation and formulate an associated inverse problem.
We determine fractional power type nonlinearities from the exterior partial measurements of the Dirichlet-to-Neumann map. Our arguments are based on a first order linearization as well as the parabolic Runge approximation property.

\section{Introduction}
Non-local equations involving the fractional Laplacian $(-\Delta)^s$ have attracted much attention in past decades. These equations have been widely used for description of anomalous diffusion and random processes with jumps in many disciplines such as probability theory, physics and finance. 

Correspondingly, inverse problems associated with fractional operators involving $(-\Delta)^s$ have also been extensively studied so far. The study in this direction was initiated in \cite{ghosh2020calderon}. The authors considered 
the linear fractional elliptic problem
$$((-\Delta)^s+ q)u= 0\,\,\, \text{in}\,\,\Omega,\qquad
u|_{\Omega_e}= g$$
where $0< s< 1$. $\Omega\subset \mathbb{R}^n$ is a bounded domain with smooth boundary and $\Omega_e:= \mathbb{R}^n\setminus \bar\Omega$. They showed that the electric potential $q$ in $\Omega$ 
can be determined from 
the partial knowledge of the Dirichlet-to-Neumann map 
$$\Lambda_q: g\to (-\Delta)^su|_{\Omega_e}.$$
See \cite{bhattacharyya2021inverse, cekic2020calderon, covi2020inverse,  covi2020higher, ghosh2017calderon, ghosh2020uniqueness, li2021determining, ruland2020fractional} for further results for linear fractional elliptic inverse problems. 

In recent years, inverse problems associated with linear fractional parabolic operators have been studied as well. See \cite{lai2020calderon} for
the study of an inverse problem associated with the fractional operator 
$(\partial_t- \Delta)^s+ q$. See \cite{li2021fractional} for
the study of an inverse problem associated with a fractional parabolic operator involving time-dependent magnetic and electric potentials.

In this paper, we will study a fractional parabolic operator involving fractional power type nonlinearities and an associated inverse problem. In fact, the combination of fractional diffusion and power type nonlinearities has given rise to interesting mathematical models which have a number of scientific applications. The well-known equation
$$\partial_t u+ (-\Delta)^s(|u|^{m-1}u)= 0,$$
for instance, describes anomalous diffusion through porous media. See \cite{vazquez2014recent} for more details.

Here we will focus on another semi-linear fractional parabolic problem
\begin{equation}\label{fracparafrac}
\left\{
\begin{aligned}
\partial_t u+ (-\Delta)^s u + a(x, t, u)&= 0
\quad \,\,\,\Omega\times (0, T)\\
u&= g\quad \,\,\,\Omega_e\times (0, T)\\
u&= 0\quad \,\,\,\mathbb{R}^n\times \{0\}
\end{aligned}
\right.
\end{equation}
where the nonlinearity satisfies
\begin{equation}\label{fracpower}
a(x, t, z)= \sum^m_{k=1}a_k(x, t)|z|^{b_k}z,
 \end{equation}
$0\leq a_k\in C(\bar{\Omega}\times [0, T])$ and the powers $0< b_1<\cdots< b_m$ are not necessarily integers.

We will show (\ref{fracparafrac}) is well-posed at least for regular and small enough $g$, which enables us to define the parabolic Dirichlet-to-Neumann map
\begin{equation}\label{DNmap}
\Lambda_{a}g:= (-\Delta)^s u|_{\Omega_e\times (0, T)}.
\end{equation}
Our goal is to determine the nonlinearity from the exterior partial measurements of $\Lambda_{a}$. 

We mention that inverse problems associated with nonlinear equations have been extensively studied as well. See \cite{sun1997inverse} for inverse problems associated with quasi-linear elliptic 
equations. See \cite{kian2020recovery} for inverse problems associated with parabolic equations involving general semi-linear terms. See \cite{kurylev2018inverse} for inverse problems associated with nonlinear hyperbolic equations. We also remark that the higher order linearization technique has been commonly applied to determine the full nonlinearity in dealing with inverse problems associated with power type nonlinear equations. See \cite{krupchyk2020remark, lassas2021inverse, liimatainen2020inverse} for this approach for elliptic problems. See \cite{lai2020inverse, lai2021inverse} for the higher order linearization approach for fractional elliptic problems.

In this paper, we will only use a first order linearization to determine all the coefficients $a_k$ based on the Runge approximation property of the linear fractional parabolic operator. This approach can be viewed as a parabolic analogue of the one in \cite{li2021inverse}.

The following theorem is the main result in this paper.
\begin{thm}
Let $W_1, W_2\subset \Omega_e$ be open. Let $a^{(1)}, a^{(2)}$ be nonlinearities of the form (\ref{fracpower}). Suppose
$$\Lambda_{a^{(1)}}g|_{W_2\times (0, T)}
= \Lambda_{a^{(2)}}g|_{W_2\times (0, T)}$$
for small $g\in C^\infty_c(W_1\times (0, T))$.
Then $a^{(1)}_k= a^{(2)}_k$ in $\Omega\times(0, T)$, $k= 1,\cdots, m$.
\end{thm}

The rest of this paper is organized in the following way. In Section 2, we summarize the background knowledge. We show the well-posedness of  (\ref{fracparafrac}) in Section 3. We prove the main theorem in Section 4.~\\

\noindent \textbf{Acknowledgement.} The author would like to thank Professor Gunther Uhlmann for suggesting the problem and for helpful discussions.

\section{Preliminaries}

Throughout this paper we use the following notations.
\begin{itemize}
\item Fix the space dimension $n\geq 2$ and 
the fractional power $0< s< 1$.

\item Fix the constant $T> 0$ and $t$ denotes the time variable.

\item $\Omega$ denotes a bounded domain with smooth boundary and
$\Omega_e:= \mathbb{R}^n\setminus\bar{\Omega}$.

\item Let $u$ be an $(n+1)$-variable function. Then 
$u(t)$ denotes the $n$-variable function $u(\cdot, t)$.

\item $c, C, C', C_1,\cdots$ denote positive constants (which may depend on some parameters). We write $C_I$ when we emphasize that the constant $C$ depends on the parameter $I$.

\item $\langle\cdot, \cdot\rangle$ denotes the distributional pairing so formally, $\langle f, g\rangle= \int fg$.
\end{itemize}

\subsection{Function spaces}

Throughout this paper we refer all function spaces to real-valued function spaces.

For $r\in \mathbb{R}$, we have Sobolev spaces
$$H^r(\mathbb{R}^n):= \{f\in \mathcal{S}'(\mathbb{R}^n):
\int_{\mathbb{R}^n} (1+\vert \xi\vert^2)^r\vert \mathcal{F}f(\xi)\vert^2d\xi<\infty\}$$
where $\mathcal{F}$ is the Fourier transform and $\mathcal{S}'(\mathbb{R}^n)$ is the space of temperate distributions.

Let $U\subset \mathbb{R}^n$ be an open. 
We use $C^s(U)= C^{0, s}(U)$ to denote 
the H\"older space equipped with the standard norm
$$||f||_{C^s(U)}:= ||f||_{L^\infty(U)}+ 
\sup_{x\neq y, x,y\in U}\frac{|f(x)- f(y)|}{|x-y|^s}.$$
For $k\in \mathbb{N}$, the norm
$||\cdot||_{C^k(\mathbb{R}^n)}$ is defined by
$$||f||_{C^k(\mathbb{R}^n)}= 
\sum_{|\alpha|\leq k}||\partial^\alpha f||_{L^\infty(\mathbb{R}^n)}.$$

Let $X$ be a Banach space. We use $C([0, T]; X)$
to denote the space consisting of the corresponding Banach space-valued  continuous functions on $[0, T]$. 
$L^p(0, T; X)$ denotes the space consisting of the corresponding Banach space-valued $L^p$ functions, equipped with the standard norm
$$||u||_{L^p(0, T; X)}:= (\int^T_0||u(t)||^p_X\,dt)^{1/p}.$$

\subsection{The fractional Laplacian and the associated semi-group estimates}
For regular functions defined in $\mathbb{R}^n$, the non-local operator $(-\Delta)^s$ is given 
by the pointwise definition 
$$(-\Delta)^su(x):= 
c_{n,s}\lim_{\epsilon\to 0^+}
\int_{\mathbb{R}^n\setminus B_\epsilon(x)}\frac{u(x)-u(y)}{|x-y|^{n+2s}}\,dy$$
as well as the the equivalent bilinear form definition
$$\langle (-\Delta)^su, v\rangle:=  c'_{n,s}\int_{\mathbb{R}^n}\int_{\mathbb{R}^n}\frac
{(u(x)-u(y))(v(x)-v(y))}{|x-y|^{n+2s}}\,dxdy.$$

Let $\{S_\Omega(t)\}_{t\geq 0}$ be the semi-group on $L^2(\Omega)$ associated with the problem
\begin{equation}\label{bddsemigroup}
\left\{
\begin{aligned}
\partial_t u+ (-\Delta)^s u &= 0\quad \,\,\,\Omega\times \mathbb{R}_+\\
u&= 0\quad \,\,\,\Omega_e\times \mathbb{R}_+\\
u&= f\quad \,\,\,\mathbb{R}^n\times \{0\}.
\end{aligned}
\right.
\end{equation}
It has been proved in \cite{fernandez2016boundary} that for $f\in L^2(\Omega)$ and $t> 0$, $S_\Omega(t)f\in C^s(\mathbb{R}^n)$ and it is smooth in $t$. See \cite{felsinger2013local, silvestre2012holder} for more H\"older regularity results for fractional parabolic equations in bounded domains. Let $\{S(t)\}_{t\geq 0}$ be the semi-group associated with the problem
\begin{equation}\label{semigroup}
\left\{
\begin{aligned}
\partial_t u+ (-\Delta)^s u &= 0\quad \,\,\,\mathbb{R}^n\times \mathbb{R}_+\\
u&= f\quad \,\,\,\mathbb{R}^n\times \{0\}.
\end{aligned}
\right.
\end{equation}
By taking the Fourier transform, we have
$$S(t)f= K(\cdot, t)* f,\qquad K(x, t)= \int_{\mathbb{R}^n} e^{ix\cdot\xi- t|\xi|^{2s}}\,d\xi.$$
It has been shown that the associated kernel
$K(x, t)$ is smooth and positive in $\mathbb{R}^n\times \mathbb{R}_+$. See \cite{greco2016existence, miao2008well} for more details. The following estimate was given in Lemma 3.1 in \cite{miao2008well}.
\begin{prop}
Suppose $2\leq r\leq p\leq\infty$ and $f\in L^r(\mathbb{R}^n)$.
Then
$$||S(t)f||_{L^p(\mathbb{R}^n)}\leq Ct^{-\frac{n}{2s}(\frac{1}{r}- \frac{1}{p})}||f||_{L^r(\mathbb{R}^n)}.$$
\end{prop}

Now we identify $f\in L^r(\Omega)$ with its zero extension as a function in $L^r(\mathbb{R}^n)$ and
consider $\tilde{u}:= S(t)|f|- S_\Omega(t)|f|$. Note that $\tilde{u}$ solves the fractional parabolic equation in $\Omega\times \mathbb{R}_+$, $\tilde{u}\geq 0$ in $\Omega_e\times \mathbb{R}_+$ and $\tilde{u}(0)= 0$. By the fractional parabolic maximum principle (see for instance, Lemma 3.2 in \cite{barrios2014widder}), we have
$$|S_\Omega(t)f|\leq S_\Omega(t)|f|\leq S(t)|f|$$
in $\Omega\times \mathbb{R}_+$. Hence the following estimate immediately follows from the previous one.

\begin{prop}
Suppose $2\leq r\leq p\leq\infty$ and $f\in L^r(\Omega)$.
Then
$$||S_\Omega(t)f||_{L^p(\Omega)}\leq Ct^{-\frac{n}{2s}(\frac{1}{r}- \frac{1}{p})}||f||_{L^r(\Omega)}.$$
\end{prop}

\subsection{The Runge approximation property}
The following unique continuation property of $(-\Delta)^s$ was first established in Theorem 1.2 in \cite{ghosh2020calderon}. 
\begin{prop}
Let $0< s< 1$ and $u\in H^r(\mathbb{R}^n)$ for some $r\in \mathbb{R}$. Let $W\subset \mathbb{R}^n$ be open. If $$(-\Delta)^su= u= 0\quad\text{in}\,\,W,$$
then $u= 0$ in $\mathbb{R}^n$.
\end{prop}
This result was later extended for the fractional Laplacian when the fractional power belongs to $(-n/2, \infty)\setminus \mathbb{Z}$. See \cite{covi2020unique} for more details.

Based on this unique continuation property, the following parabolic Runge approximation property was established in Theorem 2 in \cite{ruland2020quantitative}.
\begin{prop}
Let $W\subset \Omega_e$ be open. Then the set
$$S:= \{u_g|_{\Omega\times (0, T)}: g\in C^\infty_c(W\times (0, T))\}$$
is dense in $L^2(\Omega\times (0, T))$. Here $u_g$ is the solution corresponding to the exterior data $g$ of the linear fractional parabolic problem
\begin{equation}\label{linearpara}
\left\{
\begin{aligned}
\partial_t u+ (-\Delta)^s u&= 0\quad \,\,\,\Omega\times (0, T)\\
u&= g\quad \,\,\,\Omega_e\times (0, T)\\
u&= 0\quad \,\,\,\mathbb{R}^n\times \{0\}.
\end{aligned}
\right.
\end{equation}
\end{prop}
This result was later extended for more general fractional parabolic operators (see Proposition 4.2 in \cite{li2021fractional}). See \cite{dipierro2019local} for $C^k$-type approximation results for fractional heat equations.

We remark that both properties are typical non-local phenomenons. They enable us to obtain strong results for inverse problems associated with fractional operators.

\section{The forward problem}
\subsection{Existence}
We say $(q, p, r)$ is admissible if $1< r\leq p< nr/(n-2s)$ and 
$$\frac{1}{q}= \frac{n}{2s}(\frac{1}{r}- \frac{1}{p}).$$
We also define
$$(Gf)(x, t):= \int^t_0 S_\Omega(t-\tau)f(x, \tau)\,d\tau.$$
We first prove the following two estimates, which are analogues of 
estimates established in Lemma 3.3 in \cite{miao2008well}.
\begin{prop}
Suppose $b>0$, $r> \frac{nb}{2s}$, $2(b+1)\leq  p< r(b+1)$
and $(q, p, r)$ is admissible. Then
$$||Gf||_{L^\infty(0, T; L^r(\Omega))}\leq CT^{1-\frac{nb}{2rs}}||f||_{L^{\frac{q}{b+1}}(0, T; L^\frac{p}{b+1}(\Omega))},$$
$$||Gf||_{L^q(0, T; L^p(\Omega))}\leq CT^{1-\frac{nb}{2rs}}||f||_{L^{\frac{q}{b+1}}(0, T; L^\frac{p}{b+1}(\Omega))}.$$
\end{prop}
\begin{proof}
First we note that $\frac{b+1}{q}< \frac{nb}{2rs}< 1$ since $(q, p, r)$ is admissible and $p< r(b+ 1)$.

To prove the first inequality, 
we apply Proposition 2.2 and H\"older inequality to obtain
$$||(Gf)(t)||_{L^r(\Omega)}\leq 
\int^t_0 ||S_\Omega(t-\tau)f(\tau, x)||_{L^r(\Omega)}\,d\tau$$
$$\leq C\int^t_0(t-\tau)^{-\frac{n}{2s}(\frac{b+1}{p}- \frac{1}{r})}
||f(\tau)||_{L^\frac{p}{b+1}(\Omega)}\,d\tau$$
$$\leq C(\int^t_0(t-\tau)^{-\frac{n}{2s}(\frac{b+1}{p}- \frac{1}{r})l}\,d\tau)^\frac{1}{l}||f||_{L^{\frac{q}{b+1}}(0, T; L^\frac{p}{b+1}(\Omega))}$$
$$= C't^{\frac{1}{l}-\frac{n}{2s}(\frac{b+1}{p}- \frac{1}{r})} ||f||_{L^{\frac{q}{b+1}}(0, T; L^\frac{p}{b+1}(\Omega))}
= C't^{1-\frac{nb}{2rs}}||f||_{L^{\frac{q}{b+1}}(0, T; L^\frac{p}{b+1}(\Omega))}$$
for each $t\in (0, T)$ where $l$ satisfies $\frac{1}{l}+ \frac{b+1}{q}= 1$. Hence we have
$$||(Gf)(t)||_{L^r(\Omega)}\leq C'T^{1-\frac{nb}{2rs}}||f||_{L^{\frac{q}{b+1}}(0, T; L^\frac{p}{b+1}(\Omega))}$$
for $t\in (0, T)$ since $1-\frac{nb}{2rs}> 0$ so the first inequality has been proved
.

To prove the second inequality, we apply Proposition 2.2 to obtain
$$||Gf||_{L^q(0, T; L^p(\Omega))}\leq 
||\int^t_0 ||S_\Omega(t-\tau)f(\tau, x)||_{L^p(\Omega)}\,d\tau||_{L^q}$$
$$\leq C||\int^t_0(t-\tau)^{-\frac{n}{2s}(\frac{b+1}{p}- \frac{1}{p})}
||f(\tau)||_{L^\frac{p}{b+1}(\Omega)}\,d\tau||_{L^q}$$

By Young's convolution inequality, for $t\in (0, T)$ we have
$$||\int^t_0(t-\tau)^{-\frac{nb}{2sp}}
||f(\tau)||_{L^\frac{p}{b+1}(\Omega)}\,d\tau||_{L^q}
\leq ||\tilde{f}* \tilde{g}||_{L^q}\leq ||\tilde{f}||_{L^\frac{q}{b+1}}||\tilde{g}||_{L^l}$$
$$=(\int^T_0 t^{-\frac{nb}{2ps}\cdot l}\,dt)^{\frac{1}{l}}||f||_{L^{\frac{q}{b+1}}(0, T; L^\frac{p}{b+1}(\Omega))}
= C'T^{1-\frac{nb}{2rs}}
||f||_{L^{\frac{q}{b+1}}(0, T; L^\frac{p}{b+1}(\Omega))}$$
where $l$ satisfies $1+ \frac{1}{q}= \frac{1+b}{q}+ \frac{1}{l}$, $\tilde{f}, \tilde{g}$
are defined to be zeros outside $(0, T)$ and for $t\in (0, T)$,
$$\tilde{f}(t):= ||f(t)||_{L^\frac{p}{b+1}(\Omega)},\quad 
\tilde{g}(t):= t^{-\frac{nb}{2ps}}.$$

\end{proof}

Now we consider the semi-linear fractional problem
\begin{equation}\label{nonhomopara}
\left\{
\begin{aligned}
\partial_t u+ (-\Delta)^s u + a(x, t, u)&= g\quad \,\,\,\Omega\times (0, T)\\
u&= 0\quad \,\,\,\Omega_e\times (0, T)\\
u&= 0\quad \,\,\,\mathbb{R}^n\times \{0\}
\end{aligned}
\right.
\end{equation}
where the nonlinearity has the form (\ref{fracpower}). 

\begin{remark}
Note that for fixed powers $0< b_1<\cdots< b_m$, we can choose a large $r$ s.t. $r> \max\{\frac{nb_m}{2s}, 2(b_m+ 1)\}$. For this chosen $r$, we can choose $p>r$ s.t.
$\frac{p}{r}< \min\{\frac{n}{n-2s}, b_1+1\}$. Then $p$ and $r$ satisfy all conditions in Proposition 3.1 for $b= b_k$ ($k= 1,\cdots,m$). 
\end{remark}

From now on we fix our choices for $p$ and $r$.

We define the nonlinear map 
$$(Fu)(x, t):= \int^t_0 S_\Omega(t-\tau)(g(x, \tau)- a(x, \tau, u(x, \tau)))\,d\tau$$
so the equation in (\ref{nonhomopara}) can be converted to the integral equation $u= Fu$ (see for instance, Section 15.1 in \cite{TaylorPDE3}).
\begin{prop}
For sufficiently small $g$, (\ref{nonhomopara}) has a solution in the space
$$X:= C([0, T]; L^r(\Omega))\cap L^q(0, T; L^p(\Omega)).$$
Here $X$ is equipped with the norm 
$$||\cdot||_X:= ||\cdot||_{L^q(0, T; L^p(\Omega))}+
||\cdot||_{L^\infty(0, T; L^r(\Omega))}.$$
\end{prop}
\begin{proof}
We will show $F$ is a contraction map on the complete metric space
$$X_M:=\{u\in X: ||u||_X\leq M\}.$$
Here the constant $M$ will be determined later, which depends on the norm of $g$.

Note that for $u\in X_M$, by Proposition 3.1 we have
$$||Fu||_X\leq CT^{1-\frac{nb_m}{2rs}}||g(x, t)- a(x,t,u(x,t))||_{L^{\frac{q}{b_m+1}}(0, T; L^\frac{p}{b_m+1}(\Omega))}.$$
Using H\"older inequalities we obtain
$$||a(x,t,u(x,t))||_{L^{\frac{q}{b_m+1}}(0, T; L^\frac{p}{b_m+1}(\Omega))}
\leq C'_{a, \Omega}||\sum^m_{k=1} ||u(t)||^{b_k+1}_{L^p(\Omega)}||_{L^\frac{q}{b_m+1}}$$
$$\leq C'_{a, \Omega, T}\sum^m_{k=1} ||u||^{b_k+1}_{L^q(0, T; L^p(\Omega))}
\leq C'_{a, \Omega, T}\sum^m_{k=1} M^{b_k+1}.$$
Hence we have
$$||Fu||_X\leq C''_{a, \Omega, T}(||g||_{L^{\frac{q}{b_m+1}}(0, T; L^\frac{p}{b_m+1}(\Omega))}+ \sum^m_{k=1} M^{b_k+1}).$$

Also note that
$$||z_2|^{b_k}z_2- |z_1|^{b_k}z_1|\leq C_k|z_2- z_1|(|z_1|^{b_k}+ |z_2|^{b_k})$$
so for $u_1, u_2\in X_M$, by Proposition 3.1 and H\"older inequalities
we get that 
$$||Fu_1- Fu_2||_X\leq CT^{1-\frac{nb_m}{2rs}}||a(x,t,u_2(x,t))- a(x,t,u_1(x,t))||_{L^{\frac{q}{b_m+1}}(0, T; L^\frac{p}{b_m+1}(\Omega))}$$
$$\leq C_{a, \Omega, T}||u_1- u_2||_{L^q(0, T; L^p(\Omega))}
\sum^m_{k=1}(||u_1||^{b_k}_{L^q(0, T; L^p(\Omega))}
+ ||u_2||^{b_k}_{L^q(0, T; L^p(\Omega))})$$
$$\leq 2C_{a, \Omega, T}||u_1- u_2||_X\sum^m_{k=1} M^{b_k}.$$

Now for sufficiently small $g$, we choose $M$ to be $||g||^{1/2}_{L^{\frac{q}{b+1}}(0, T; L^\frac{p}{b+1}(\Omega))}$. 

Then
$||Fu||_X\leq M$ for $u\in X_M$ and $||Fu_1- Fu_2||_X\leq \frac{1}{2}||u_1- u_2||_X$ for $u_1, u_2\in X_M$. We can estimate $(Fu)(x, t+h)- (Fu)(x, t)$ in similar ways to show the continuity of $Fu$ in $t$. Hence $F$ has a fix point
in $X_M$ by the fixed-point theorem.
\end{proof}

Note that $||(-\Delta)^s f||_{L^\infty(\mathbb{R}^n)}\leq C||f||_{C^2(\mathbb{R}^n)}$ for $f\in C^\infty_c(\mathbb{R}^n)$
(see for instance, Lemma 3.3 in \cite{li2021inverse}) so by Proposition 3.2 the problem
\begin{equation}
\left\{
\begin{aligned}
\partial_t w+ (-\Delta)^s w + a(x, t, w)&= -(-\Delta)^s g\quad \,\,\,\Omega\times (0, T)\\
w&= 0\qquad\qquad \,\,\,\,\,\,\,\Omega_e\times (0, T)\\
w&= 0\qquad\qquad \,\,\,\,\,\,\,\mathbb{R}^n\times \{0\}
\end{aligned}
\right.
\end{equation}
has a solution in $X$ for small $g\in C^\infty_c(\Omega_e\times (0, T))$. Then $u:= w+ g$ gives a solution of (\ref{fracparafrac}).

\subsection{$L^\infty$-estimate and uniqueness}
We first prove the following $L^\infty$-estimate, which will be used in the first order linearization later.

\begin{prop}
Suppose $u$ is a solution of the problem
\begin{equation}\label{fgpara}
\left\{
\begin{aligned}
\partial_t u+ (-\Delta)^s u + a(x, t, u)&= f
\quad \,\,\,\Omega\times (0, T)\\
u&= g\quad \,\,\,\Omega_e\times (0, T)\\
u&= 0\quad \,\,\,\mathbb{R}^n\times \{0\}.
\end{aligned}
\right.
\end{equation}
\end{prop}
Then we have
$$||u||_{L^\infty}\leq T||f||_{L^\infty(\Omega\times (0, T))}
+ ||g||_{L^\infty(\Omega_e\times (0, T))}.$$
\begin{proof}
We fix $\phi\in C^\infty_c(\mathbb{R}^n)$ s.t. $0\leq \phi\leq 1$
and $\phi= 1$ on $\bar\Omega \cup \bar W$. We define
$$\tilde{\phi}(x, t):= (||f||_{L^\infty(\Omega\times (0, T))}t
+ ||g||_{L^\infty(\Omega_e\times (0, T))})\phi(x).$$
Clearly $(-\Delta)^s\phi\geq 0$ in $\Omega$ from the pointwise definition 
of $(-\Delta)^s$ so $$\partial_t\tilde{\phi}+(-\Delta)^s\tilde{\phi}+ a(x, t, \tilde{\phi})\geq ||f||_{L^\infty(\Omega\times (0, T))}$$
in $\Omega\times (0, T)$. Now we consider $\tilde{u}:= \tilde{\phi}- u$.
Note that $\tilde{u}\geq 0$ in 
$\Omega_e\times (0, T)$, $\tilde{u}\geq 0$ at $t= 0$ and 
$$\partial_t\tilde{u}+(-\Delta)^s\tilde{u}+ a(x, t, \tilde{\phi})
- a(x, t, u)\geq 0$$
in $\Omega\times (0, T)$. Write $\tilde{u}= \tilde{u}^+ - \tilde{u}^-$
where $\tilde{u}^{\pm}= \max\{\pm \tilde{u}, 0\}$. Then $\tilde{u}^-= 0$ in $\Omega_e\times (0, T)$. 

Now we define $E(t):= ||\tilde{u}^-(t)||^2_{L^2(\Omega)}
\geq 0$. Then 
$E(0)= 0$ and
$$E'(t)= -2\langle \partial_t\tilde{u}(t), \tilde{u}^-(t)\rangle
\leq 2\langle (-\Delta)^s\tilde{u}(t), \tilde{u}^-(t)\rangle
+ 2\langle a(x, t, \tilde{\phi})- a(x, t, u), \tilde{u}^-(t)\rangle.$$
Note that
$$\langle (-\Delta)^s\tilde{u}(t), \tilde{u}^-(t)\rangle=
\iint\frac{(\tilde{u}(t)(x)- \tilde{u}(t)(y))(\tilde{u}^-(t)(x)- \tilde{u}^-(t)(y))}{|x-y|^{n+2s}}dxdy.$$
Since $(\tilde{u}^+(t)(x)- \tilde{u}^+(t)(y))(\tilde{u}^-(t)(x)- \tilde{u}^-(t)(y))\leq 0$, we have 
$$\langle (-\Delta)^s\tilde{u}(t), \tilde{u}^-(t)\rangle
\leq -\iint\frac{(\tilde{u}^-(t)(x)- \tilde{u}^-(t)(y))^2}{|x-y|^{n+2s}}dxdy\leq 0.$$
Also note that $a(x, t, \tilde{\phi})- a(x, t, u)$ have the same sign as $\tilde{u}$. Hence
we have $E'(t)\leq 0$ so the only possibility is $E(t)= 0$ and thus $\tilde{\phi}\geq u$ in $\Omega\times (0, T)$. 

Also note that $a(x, t, \tilde{\phi})+ a(x, t, u)$ have the same sign as $\tilde{\phi}+ u$ so similarly we can consider $\tilde{u}:= \tilde{\phi}+ u$ and show $\tilde{\phi}\geq -u$ in $\Omega\times (0, T)$. Hence we have
$|u|\leq \tilde{\phi}$ in $\Omega\times (0, T)$.
\end{proof}

\begin{remark}
In fact, we can use similar arguments to show the uniqueness of solutions of (\ref{fgpara}). 
Suppose $u_1, u_2$ are two solutions of (\ref{fgpara}). Let $\tilde{u}:= u_1- u_2$. Then
$\tilde{u}= 0$ in 
$\Omega_e\times (0, T)$, $\tilde{u}= 0$ at $t= 0$ and 
$$\partial_t\tilde{u}+(-\Delta)^s\tilde{u}+ a(x, t, u_1)
- a(x, t, u_2)= 0$$
in $\Omega\times (0, T)$. Since $a(x, t, u_1)
- a(x, t, u_2)$ has the same sign as $\tilde{u}$, we can show that 
$$E'(t)= -2\langle \partial_t\tilde{u}(t), \tilde{u}^-(t)\rangle
= 2\langle (-\Delta)^s\tilde{u}(t), \tilde{u}^-(t)\rangle
+ 2\langle a(x, t, u_1)- a(x, t, u_2), \tilde{u}^-(t)\rangle\leq 0$$
where
$E(t):= ||\tilde{u}^-(t)||^2_{L^2(\Omega)}$ as before to conclude that $u_1\geq u_2$.
 Similarly we consider $\tilde{u}= u_2- u_1$ to prove $u_2\geq u_1$. Hence $u_1= u_2$. 
\end{remark}

In particular, we conclude that (\ref{fracparafrac}) has a unique solution.

\section{The inverse problem}
The well-posedness result in the previous section ensures that the Dirichlet-to-Neumann map
$$\Lambda_{a}g:= (-\Delta)^s u|_{\Omega_e\times (0, T)}$$
is well-defined at least for regular and small $g$.

\begin{remark}
It has been shown that the knowledge of $\Lambda_{a}$ is equivalent to the knowledge of the non-local Neumann operator $\mathcal{N}_su$ (see for instance, \cite{ghosh2020calderon}), which is defined by
$$\mathcal{N}_su(t)(x):= 
c_{n,s}
\int_{\Omega}\frac{u(t)(x)-u(t)(y)}{|x-y|^{n+2s}}\,dy,\quad 
x\in \Omega_e, t\in (0, T).$$ 
Recall that for the inverse problem for the classical semi-linear parabolic operator (see for instance, \cite{kian2020recovery}), the 
Dirichlet-to-Neumann map was defined  by
$$\Lambda_{a}: g\to \partial_\nu u|_{\partial\Omega\times (0, T)}$$
where $\partial_\nu u$ is the classical Neumann derivative associated with the 
initial boundary value problem
\begin{equation*}
\left\{
\begin{aligned}
\partial_t u -\Delta u + a(x, t, u)&= 0\quad \,\,\,\Omega\times (0, T)\\
u&= g\,\,\,\,\, \partial\Omega\times (0, T)\\
u&= 0\quad \,\,\,\Omega\times \{0\}.
\end{aligned}
\right.
\end{equation*}
Hence our inverse problem can be viewed as a natural non-local analogue of the classical problem.
\end{remark}

We will apply the linearization scheme, which enables us to use the Runge approximation property for the linear fractional parabolic operator to deal with the inverse problem for the 
semi-linear fractional parabolic operator.
 
\subsection{Linearization}
For $g\in C^\infty_c(\Omega_e\times (0, T))$ and small $\lambda> 0$, we use 
$u_g$ to denote the solution of the linear problem (\ref{linearpara}) and we use $u_{\lambda, g}$ to denote the solution of the semi-linear problem
\begin{equation}\label{lambdafrac}
\left\{
\begin{aligned}
\partial_t u+ (-\Delta)^s u + a(x, t, u)&= 0\quad \,\,\,\Omega\times (0, T)\\
u&= \lambda g\,\,\,\,\, \Omega_e\times (0, T)\\
u&= 0\quad \,\,\,\mathbb{R}^n\times \{0\}.
\end{aligned}
\right.
\end{equation}

\begin{prop}
Let $v_{\lambda, g}:= u_g- \frac{u_{\lambda, g}}{\lambda}$. Then $\lim_{\lambda\to 0}v_{\lambda, g}= 0$
in $L^\infty(\Omega\times (0, T))$.
\end{prop}
\begin{proof}
Note that $v_{\lambda, g}= 0$ in $\Omega_e\times (0, T)$ and
$$\partial_t v_{\lambda, g}+ (-\Delta)^s v_{\lambda, g}= \frac{1}{\lambda}
a(x, t, u_{\lambda, g})$$
in $\Omega\times (0, T)$. By the Proposition 3.3, we have
$$||v_{\lambda, g}||_{L^\infty}\leq \frac{T}{\lambda}
||a(x, t, u_{\lambda, g})||_{L^\infty(\Omega\times (0, T))}$$
and $||u_{\lambda, g}||_{L^\infty}\leq 
\lambda||g||_{L^\infty(\Omega_e\times (0, T))}$ so
$$||v_{\lambda, g}||_{L^\infty}\leq 
T\sum^m_{k=1}\lambda^{b_k}||a_k(x,t)||_{L^\infty}||g||^{b_k+1}_{L^\infty(\Omega_e\times (0, T))},$$
which implies $||v_{\lambda, g}||_{L^\infty}\to 0$ as $\lambda\to 0$.
\end{proof}

\subsection{Proof of the main theorem}
Now we are ready to prove Theorem 1.1.
\begin{proof}
Since we have the assumption
$$\Lambda_{a^{(1)}}g|_{W_2\times (0, T)}
= \Lambda_{a^{(2)}}g|_{W_2\times (0, T)}$$
for small $g\in C^\infty_c(W_1\times (0, T))$ and 
$u^{(1)}_{\lambda, g}= u^{(2)}_{\lambda, g}= \lambda g$ in $\Omega_e\times (0, T)$,
for each $t$ we have
$$(-\Delta)^s(u^{(1)}_{\lambda, g}- u^{(2)}_{\lambda, g})(t)|_{W_2}= 0, \qquad (u^{(1)}_{\lambda, g}- u^{(2)}_{\lambda, g})(t))|_{\Omega_e}= 0$$
for $g\in C^\infty_c(W_1\times (0, T))$ and small $\lambda> 0$ so Proposition 2.3 implies 
$u^{(1)}_{\lambda, g}= u^{(2)}_{\lambda, g}=:  u_{\lambda, g}$ in $\mathbb{R}^n\times (0, T)$. Hence we have
$$a^{(1)}(x, t, u_{\lambda, g}(x,t))
= a^{(2)}(x, t, u_{\lambda, g}(x,t))$$
\begin{equation}\label{aRid}
(a^{(1)}_1(x, t)- a^{(2)}_1(x, t))|u_{\lambda, g}|^{b_1}u_{\lambda, g}
= R^{(2)}_1(x, t, u_{\lambda, g})
- R^{(1)}_1(x, t, u_{\lambda, g})
\end{equation}
in $\Omega\times (0, T)$ where
$$R^{(i)}_j(x, t, z):= \sum^m_{k=j+1}a^{(i)}_k(x, t)|z|^{b_k}z.$$

Now note that
$$|||a^{(1)}_1(x,t)- a^{(2)}_1(x,t)|^\frac{1}{b_1+1}||_{L^2(\Omega\times 
(0, T))}$$
$$\leq |||a^{(1)}_1(x, t)- a^{(2)}_1(x, t)|^\frac{1}{b_1+1}(1- \frac{u_{\lambda, g}}{\lambda})||_{L^2(\Omega\times (0, T)}$$
$$+ \frac{1}{\lambda}|||a^{(1)}_1(x,t)- a^{(2)}_1(x,t)|^\frac{1}{b_1+1}
u_{\lambda, g}||_{L^2(\Omega\times (0, T))}$$
$$\leq |||a^{(1)}_1(x, t)- a^{(2)}_1(x, t)|^\frac{1}{b_1+1}||_{L^\infty}||1- \frac{u_{\lambda, g}}{\lambda}||_{L^2(\Omega\times (0, T))}$$
\begin{equation}\label{2inftyest}
+ \frac{1}{\lambda}|||a^{(1)}_1(x,t)- a^{(2)}_1(x,t)|^\frac{1}{b_1+1}
u_{\lambda, g}||_{L^2(\Omega\times (0, T))}.
\end{equation}
For given $\delta> 0$, by Proposition 2.4 we can choose 
$g\in C^\infty_c(W_1\times (0, T))$ s.t. 
$$||1- u_g||_{L^2(\Omega\times (0, T))}\leq \delta$$ 
and for this chosen $g$, we have 
\begin{equation}\label{Rungelinear}
||1- \frac{u_{\lambda, g}}{\lambda}||_{L^2(\Omega\times (0, T))}\leq 2\delta
\end{equation}
for small $\lambda$ by Proposition 4.1. Since Proposition 3.3 implies that
$$||u_{\lambda, g}||_{L^\infty}\leq 
\lambda||g||_{L^\infty(\Omega_e\times (0, T))},$$
by (\ref{aRid}) we have
$$\frac{1}{\lambda}|||a^{(1)}_1(x,t)- a^{(2)}_1(x,t)|^\frac{1}{b_1+1}
u_{\lambda, g}||_{L^2(\Omega\times (0, T))}$$
$$\leq 
\frac{C}{\lambda}|||R^{(2)}_1(x, t, u_{\lambda, g})
- R^{(1)}_1(x, t, u_{\lambda, g})|^{\frac{1}{b_1+1}}||_{L^\infty(\Omega\times (0, T))}$$
$$\leq C'(\sum^m_{k=2}\lambda^{\frac{b_k-b_1}{b_1+1}}(||a^{(1)}_k||^
\frac{1}{b_1+1}_{L^\infty}+||a^{(2)}_k||^
\frac{1}{b_1+1}_{L^\infty})||g||^\frac{b_k+1}{b_1+1}_{L^\infty}).$$
This inequality implies
$$\frac{1}{\lambda}|||a^{(1)}_1(x,t)- a^{(2)}_1(x,t)|^\frac{1}{b_1+1}
u_{\lambda, g}||_{L^2(\Omega\times (0, T))}\to 0$$
as $\lambda\to 0$. Then by (\ref{2inftyest}) and (\ref{Rungelinear}) we obtain
$$|||a^{(1)}_1(x,t)- a^{(2)}_1(x,t)|^\frac{1}{b_1+1}||_{L^2(\Omega\times (0, T))}\leq 
2\delta |||a^{(1)}_1(x,t)- a^{(2)}_1(x,t)|^\frac{1}{b_1+1}||_{L^\infty(\Omega\times (0, T))}.$$
Now we conclude that $a^{(1)}_1= a^{(2)}_1$ since $\delta$ is arbitrary.

Iteratively, once we have shown $a^{(1)}_j= a^{(2)}_j$ ($1\leq j\leq m'-1$), we have 
$$(a^{(1)}_{m'}(x, t)- a^{(2)}_{m'}(x, t))|u_{\lambda, g}|^{b_{m'}}
u_{\lambda, g}
= R^{(2)}_{m'}(x, t, u_{\lambda, g})
- R^{(1)}_{m'}(x, t, u_{\lambda, g})$$
in $\Omega\times (0, T)$. Also note that
$$|||a^{(1)}_{m'}(x,t)- a^{(2)}_{m'}(x,t)|^\frac{1}{b_{m'}+1}||_{L^2(\Omega\times 
(0, T))}$$
$$\leq |||a^{(1)}_{m'}(x, t)- a^{(2)}_{m'}(x, t)|^\frac{1}{b_{m'}+1}||_{L^\infty}||1- \frac{u_{\lambda, g}}{\lambda}||_{L^2(\Omega\times (0, T))}$$
$$+ \frac{1}{\lambda}|||a^{(1)}_{m'}(x,t)- a^{(2)}_{m'}(x,t)|^\frac{1}{b_{m'}+1}
u_{\lambda, g}||_{L^2(\Omega\times (0, T))}.$$
For given $\delta> 0$, we can choose 
$g\in C^\infty_c(W_1\times (0, T))$ s.t. 
$$||1- \frac{u_{\lambda, g}}{\lambda}||_{L^2(\Omega\times (0, T))}\leq 2\delta$$
for small $\lambda$ and we also have
$$\frac{1}{\lambda}|||a^{(1)}_{m'}(x,t)- a^{(2)}_{m'}(x,t)|^\frac{1}{b_{m'}+1}
u_{\lambda, g}||_{L^2(\Omega\times (0, T))}$$
$$\leq 
\frac{C}{\lambda}|||R^{(2)}_{m'}(x, t, u_{\lambda, g})
- R^{(1)}_{m'}(x, t, u_{\lambda, g})|^{\frac{1}{b_{m'}+1}}||_{L^\infty(\Omega\times (0, T))}$$
$$\leq C'(\sum^m_{k={m'}+1}\lambda^{\frac{b_k-b_{m'}}{b_{m'}+1}}(||a^{(1)}_k||^
\frac{1}{b_{m'}+1}_{L^\infty}+||a^{(2)}_k||^
\frac{1}{b_{m'}+1}_{L^\infty})||g||^\frac{b_k+1}{b_{m'}+1}_{L^\infty}).$$
Now let $\lambda\to 0$. Then we get
$$|||a^{(1)}_{m'}(x,t)- a^{(2)}_{m'}(x,t)|^\frac{1}{b_{m'}+1}||_{L^2(\Omega\times (0, T))}\leq 
2\delta |||a^{(1)}_{m'}(x,t)- a^{(2)}_{m'}(x,t)|^\frac{1}{b_{m'}+1}||_{L^\infty(\Omega\times (0, T))}.$$
Now we conclude that $a^{(1)}_{m'}= a^{(2)}_{m'}$ since $\delta$ is arbitrary.
\end{proof}

\bibliographystyle{plain}
{\small\bibliography{Reference4}}
\end{document}